\documentclass[12pt]{amsart}

\usepackage[a4paper,margin=1.15in]{geometry}
\usepackage{amsmath,amssymb,amsthm,amsfonts,mathrsfs}
\usepackage{enumitem}
\usepackage[hidelinks]{hyperref}
\usepackage{mathtools}
\usepackage{cite}

\numberwithin{equation}{section}

\newtheorem{theorem}{Theorem}[section]
\newtheorem{lemma}[theorem]{Lemma}
\newtheorem{proposition}[theorem]{Proposition}
\newtheorem{corollary}[theorem]{Corollary}

\newtheorem{remark}[theorem]{Remark}
\newtheorem{hypothesis}[theorem]{Hypothesis}

\usepackage{textcase}

\title{Power Products in Elliptic Divisibility Sequences and Prime-Incidence Obstructions}
\author{Dongyeon Kym}

\address{Department of Mathematics, Dongguk University, Seoul 04620, Republic of Korea}
\email{kynteger@dgu.ac.kr}


\begin{document}
	
	\maketitle
	
	\begin{abstract}
	Let \(E/\mathbb Q\) be an elliptic curve, let \(P\in E(\mathbb Q)\) be non-torsion, and let \((D_n)\) be the associated elliptic divisibility sequence. For a fixed prime \(\rho\), we study when an arbitrary finite product
	\[
	\prod_{i=1}^k D_{n_i}
	\]
	can be a \(\rho\)-th power in \(\mathbb Q^\times\). The main result is that, under the hypothesis that \(D_1\) is divisible by \(2\) or \(3\), such product relations impose rigid restrictions on the large prime divisors of the indices \(n_i\). More precisely, for every \(B\ge 2\), all sufficiently large prime divisors \(\ell\) which occur as simple largest prime divisors of the indices and whose complementary cofactors are \(B\)-smooth must occur in \(\rho\)-balanced blocks. Equivalently, the corresponding prime-incidence rows over \(\mathbb F_\rho\) have pairwise disjoint supports, are linearly independent, and satisfy the packing bound
	\[
	|\Lambda^*|\le \lfloor k/\rho\rfloor .
	\]
	In particular, if \(n_i=\ell_i a_i\), where the \(\ell_i\) are sufficiently large primes and the \(a_i\) are \(B\)-smooth, then a \(\rho\)-th power product relation can hold only if each prime \(\ell\) occurs among the \(\ell_i\) with multiplicity divisible by \(\rho\).
	
	The proof combines Silverman's valuation law, a fixed finite-prime-set consequence of Reynolds' finiteness theorem, and the Hasse bound. The case \(\rho=2\) gives the corresponding square-product obstruction.
	\end{abstract}
	
	\section{Introduction}
	
	Let \(E/\mathbb Q\) be an elliptic curve given by a global minimal Weierstrass equation, and let \(P\in E(\mathbb Q)\) be a non-torsion point.
	For each \(n\geq 1\), write
	\[
	x([n]P)=\frac{A_n}{D_n^2},
	\qquad A_n,D_n\in\mathbb Z,\quad D_n>0,\quad (A_n,D_n)=1.
	\]
	The sequence \((D_n)_{n\geq 1}\) is the elliptic divisibility sequence attached to \((E,P)\), introduced by Ward~\cite{Ward1948}.
	It is the elliptic analogue of the divisibility sequences arising from linear recurrences, but its local arithmetic is governed by the reduction of \(P\) modulo primes.
	The \(p\)-adic valuation law of Silverman~\cite{Silverman2005}, recalled below, is the basic local input throughout the paper.
	
	Perfect-power questions in divisibility sequences are traditionally tied to primitive divisors.
	In the classical setting this goes back to Carmichael~\cite{Carmichael1913} and culminates, for Lucas and Lehmer numbers, in the theorem of Bilu, Hanrot and Voutier~\cite{BiluHanrotVoutier2001}; see also Everest, van der Poorten, Shparlinski and Ward~\cite{EverestPoortenShparlinskiWard2003}.
	In the elliptic setting, primitive divisor questions were studied by Everest, McLaren and Ward~\cite{EverestMcLarenWard2006}, while perfect-power questions and related denominator problems were studied by Everest, Reynolds and Stevens~\cite{EverestReynoldsStevens2007}, Reynolds~\cite{Reynolds2012}, Alfaraj~\cite{Alfaraj2023}, and Nowroozi and Siksek~\cite{NowrooziSiksek2024}.
	Products of EDS terms have also been studied, notably by Hajdu, Laishram and Szikszai~\cite{HajduLaishramSzikszai2016} for products taken over arithmetic progressions.
	
	The purpose of this paper is to isolate a mechanism by which local arithmetic on the elliptic curve controls global product relations among EDS terms. We do not try to classify perfect powers among the
	individual terms \(D_n\). Instead, we start from an arbitrary prime-power product relation
	\[
	\prod_{i=1}^k D_{n_i}\in(\mathbb Q^\times)^\rho
	\]
	and ask what this relation forces on the prime divisors of the indices \(n_i\).
	
	The resulting obstruction is incidence-theoretic, but its origin is arithmetic-geometric. Outside a fixed finite set of primes, the valuation law for elliptic divisibility sequences expresses \(v_p(D_n)\) in terms of the order of the reduction of \(P\) modulo \(p\). At a prime index \(\ell\), a prime divisor of \(D_\ell\) therefore detects the occurrence of \(\ell\) among the indices \(n_i\). The product relation then imposes a congruence condition on
	the corresponding incidence vector over \(\mathbb F_\rho\).
	
	More concretely, if
	\[
	I_\ell(\mathbf n)=\{i:\ell\mid n_i\},
	\]
	and if \(p\notin S_{E,P}\) divides \(D_\ell\), then the reduction of \(P\) modulo \(p\) has exact order \(\ell\).  Hence a relation
	\[
	\prod_{i=1}^k D_{n_i}\in(\mathbb Q^\times)^\rho
	\]
	forces the congruence
	\[
	|I_\ell(\mathbf n)|v_p(D_\ell)
	+\sum_{i\in I_\ell(\mathbf n)}v_p(n_i/\ell)
	\equiv 0 \pmod{\rho}.
	\]
	This is the basic absorption congruence. In vector form, it says that prime-power product relations impose linear restrictions over \(\mathbb F_\rho\) on the incidence pattern of prime divisors among the indices. This conversion of product relations into incidence constraints is the main organising principle of the paper.
	
	To make this local obstruction effective, one needs prime divisors of \(D_\ell\) whose valuations are nonzero modulo \(\rho\).
	We obtain these from a fixed-\(S\) consequence of Reynolds' finiteness theorem.
	Under the hypothesis that \(D_1\) is divisible by \(2\) or \(3\), we show that for each fixed prime \(\rho\) and all sufficiently large prime indices
	\(\ell\),
	\[
	\operatorname{rad}_{S_{E,P},\rho}(D_\ell)>1.
	\]
	Consequently, by the valuation law at a prime index, \(D_\ell\) has a primitive prime divisor \(p\notin S_{E,P}\) such that
	\[
	\rho\nmid v_p(D_\ell).
	\]
	We call such a prime a \(\rho\)-detecting primitive divisor.
	
	Thus Reynolds' theorem enters the argument only through this fixed-\(S\), fixed-exponent consequence. Its sole role is to ensure the existence, for all sufficiently large prime indices \(\ell\), of a prime divisor of \(D_\ell\) outside \(S_{E,P}\) whose valuation is nonzero modulo \(\rho\).
	The incidence, absorption, clustering, and large-prime-gap obstructions proved below are otherwise consequences of the local valuation law and the Hasse bound.
	
	This separation of inputs is important for the structure of the argument.
	The local incidence results are proved under the sole detecting-prime input that, for all sufficiently large prime indices \(\ell\), the term \(D_\ell\) has a prime divisor outside \(S_{E,P}\) whose valuation is nonzero modulo \(\rho\).
	Reynolds' theorem is used only to verify this input under the hypothesis that \(D_1\) is divisible by \(2\) or \(3\). Thus the contribution of the paper is not a new perfect-power finiteness theorem for individual EDS terms, but rather a mechanism which converts arbitrary prime-power product relations into incidence, absorption, and packing restrictions on the prime divisors of the indices.
	
	Combining this detecting prime with the absorption congruence gives the first support obstruction.
	If
	\[
	\rho\nmid |I_\ell(\mathbf n)|,
	\]
	then, for all sufficiently large prime \(\ell\),
	\[
	\operatorname{rad}_{S_{E,P},\rho}(D_\ell)
	\mid
	\prod_{i\in I_\ell(\mathbf n)}\frac{n_i}{\ell}.
	\]
	Moreover, if \(q\mid \operatorname{rad}_{S_{E,P},\rho}(D_\ell)\), then the reduction of \(P\) modulo \(q\) has exact order \(\ell\), and therefore
	\[
	\ell\mid \#E(\mathbb F_q).
	\]
	The Hasse bound gives
	\[
	\ell \leq q+1+2\sqrt q.
	\]
	Thus local valuation information at \(D_\ell\) forces prime divisors of Hasse scale to occur inside the quotient indices.
	
	This formulation separates the present work from the usual perfect-power problem for individual EDS terms. Rather than starting with a single term \(D_n\), or with a product in a prescribed family, we start with an arbitrary \(\rho\)-th power product relation and ask how the prime divisors of the indices can occur. The relevant object is therefore the incidence matrix
	\[
	M_\Lambda(\mathbf n)
	=(\delta_{\ell\mid n_i})_{\ell\in\Lambda,\,1\leq i\leq k}
	\in \operatorname{Mat}_{|\Lambda|\times k}(\mathbb F_\rho),
	\]
	whose rows record the occurrence of large prime divisors \(\ell\) among the indices.
	The main theorem gives rank and packing restrictions on this matrix: under top-prime and smooth-cofactor hypotheses, its nonzero rows are forced to have pairwise disjoint supports, to be linearly independent over \(\mathbb F_\rho\), and to have cardinality divisible by \(\rho\).
	
	Using Reynolds' finiteness theorem under the hypothesis that \(D_1\) is divisible by \(2\) or \(3\), we obtain the following cluster-packing obstruction. Fix a prime \(\rho\) and \(B\ge 2\). Then there exists a constant
	\[
	C_{\rho,B}=C_{\rho,B}(E,P)
	\]
	such that the following holds. If
	\[
	\prod_{i=1}^k D_{n_i}\in(\mathbb Q^\times)^\rho
	\]
	and \(\Lambda\) is a finite set of primes \(\ell>C_{\rho,B}\) such that, whenever \(\ell\mid n_i\), the prime \(\ell\) is a simple largest prime divisor of \(n_i\) and \(n_i/\ell\) is \(B\)-smooth, then
	\[
	\rho\mid |I_\ell(\mathbf n)|\qquad(\ell\in\Lambda).
	\]
	Thus large top prime divisors of the indices cannot occur with arbitrary multiplicity in a \(\rho\)-th power product relation. Under the top-prime hypotheses, the corresponding nonzero incidence rows are supported on disjoint blocks, are linearly independent over \(\mathbb F_\rho\), and satisfy the packing bound
	\[
	|\Lambda^*|\le \lfloor k/\rho\rfloor .
	\]
	
	A particularly simple consequence is obtained when the indices have the form \(n_i=\ell_i a_i\), where the \(\ell_i\) are large primes and the \(a_i\) are \(B\)-smooth.  Then a prime-power product relation can
	hold only if each prime \(\ell\) occurring among the \(\ell_i\) occurs with multiplicity divisible by \(\rho\).  In particular, if the large top primes \(\ell_1,\ldots,\ell_k\) are pairwise distinct, then
	\[
	\prod_{i=1}^k D_{\ell_i a_i}\notin(\mathbb Q^\times)^\rho .
	\]
	Thus the obstruction gives a direct exclusion for a natural family of EDS product relations, not merely a reformulation of local valuation constraints.
	
	When \(\rho=2\), the same mechanism recovers a square-product obstruction.
	In this case
	\[
	\operatorname{rad}_{S_{E,P},2}(D_\ell)
	=
	\operatorname{sqf}_{S_{E,P}}(D_\ell),
	\]
	and the condition \(\rho\nmid |I_\ell(\mathbf n)|\) becomes the parity condition that \(|I_\ell(\mathbf n)|\) is odd.  Thus the square-product case is not an isolated parity phenomenon; it is the special case \(\rho=2\) of a congruence obstruction modulo a prime exponent.
	
	The paper is organised as follows.
	Section~2 recalls the valuation law for elliptic divisibility sequences and introduces power radicals outside
	a finite set of primes.
	Section~3 proves the prime-index absorption
	congruence and its incidence-vector refinement.
	Section~4 extracts from Reynolds' work the fixed-\(S\) finiteness input needed to produce \(\rho\)-detecting primitive divisors at sufficiently large prime indices.
	Section~5 combines these ingredients to prove the prime-support obstruction, the large top-prime cluster-packing theorem, the repeated top-prime obstruction, the coprime two-term large-prime-gap exclusion, and the square-product corollary.
	
	\section{Preliminaries on elliptic divisibility sequences}
	
	\subsection{Elliptic divisibility sequences and the valuation law}
	
	Let $E/\mathbb{Q}$ be an elliptic curve given by a global minimal Weierstrass equation, and let $P \in E(\mathbb{Q})$ be a non-torsion point.
	
	For each integer $n \ge 1$, write
	\[
	x([n]P) = \frac{A_n}{D_n^2}, \qquad A_n, D_n \in \mathbb{Z}, \quad D_n > 0, \quad (A_n, D_n)=1.
	\]
	Then \((D_n)_{n\ge 1}\) is the elliptic divisibility sequence (EDS) attached to \((E,P)\).
	
	\begin{theorem}[{\kern-5pt\cite[Theorem 1.1]{Silverman2005}}]\label{thm:EDS-valuation-law}
		Let \(E/\mathbb Q\) be an elliptic curve given by a global minimal Weierstrass equation, and let \(P\in E(\mathbb Q)\) be a non-torsion point. Let \((D_n)\) be the associated elliptic divisibility sequence. Then there exists a finite set of rational primes \(S_{E,P}\), depending only on the pair \((E,P)\), such that for every prime \(p\notin S_{E,P}\) the following holds.
		
		Let
		\[
		r_p:=\operatorname{ord}(\widetilde P)\in\mathbb Z_{\ge 1}
		\]
		be the order of the reduction of \(P\) in \(E(\mathbb F_p)\). Then for every integer
		\(n\ge 1\):
		\begin{enumerate}
			\item \(p\mid D_n\) if and only if \(r_p\mid n\);
			\item if \(r_p\mid n\), then
			\[
			v_p(D_n)=v_p(D_{r_p})+v_p\!\left(\frac{n}{r_p}\right).
			\]
		\end{enumerate}
	\end{theorem}
	
	For the rest of the paper, we fix such a finite set \(S_{E,P}\). We may and shall assume that \(S_{E,P}\) contains all primes of bad reduction for \(E\), all primes for which the reduction \(\widetilde P\in E(\mathbb F_p)\) is trivial, and any further primes required for the valuation law in Theorem~\ref{thm:EDS-valuation-law}.
	
	For a positive rational number \(a\), let
	\[
	\operatorname{Supp}(a):=\{p:\ v_p(a)\ne 0\}
	\]
	denote its prime support.
	
	We further enlarge \(S_{E,P}\), if necessary, so that
	\[
	\operatorname{Supp}(D_1)\subseteq S_{E,P}.
	\]
	
	For every prime \(p\notin S_{E,P}\), we write
	\[
	r_p:=\operatorname{ord}(\widetilde P)
	\]
	for the order of the reduction of \(P\) in \(E(\mathbb F_p)\).

	\begin{remark}
		The set \(S_{E,P}\) is not canonical. Throughout the paper it is fixed once and for all, and may be enlarged without affecting any of the arguments. In particular, we choose it large enough so that the valuation law of Theorem~\ref{thm:EDS-valuation-law} holds for all primes outside \(S_{E,P}\). All constants below are allowed to depend on this fixed choice of \(S_{E,P}\); since \(S_{E,P}\) has been fixed in terms of \((E,P)\), we write them as depending only on \((E,P)\).
	\end{remark}
	
	\subsection{Power classes modulo a finite set of primes}
	
	Let \(S\) be a finite set of rational primes, and let \(\rho\) be a prime number.
	For a positive integer \(a\), define
	\[
	\operatorname{rad}_{S,\rho}(a)
	:=
	\prod_{\substack{p\notin S\\ \rho\nmid v_p(a)}} p.
	\]
	We also write
	\[
	\operatorname{rad}(N):=\prod_{p\mid N}p
	\]
	for the ordinary radical of a positive integer \(N\).
	Thus \(\operatorname{rad}_{S,\rho}(a)=1\) if and only if all prime valuations of \(a\) outside \(S\) are divisible by \(\rho\).
	Equivalently,
	\[
	a\in(\mathbb Q^\times_{>0})^\rho\cdot
	\langle p:p\in S\rangle.
	\]
	
	In the special case \(\rho=2\), we write
	\[
	\operatorname{sqf}_S(a):=\operatorname{rad}_{S,2}(a).
	\]
	Thus
	\[
	\operatorname{sqf}_S(a)
	=
	\prod_{\substack{p\notin S\\ v_p(a)\equiv 1\pmod 2}}p
	\]
	is the squarefree part of \(a\) outside \(S\).
	
	\section{Prime-index absorption modulo a prime exponent}
	
	The key mechanism underlying the argument is a local congruence obstruction arising at prime indices.  Precisely, a prime-power product condition imposes valuation congruences at primes dividing \(D_\ell\).
	
	The following elementary lemma identifies the order of the reduction at primes dividing \(D_\ell\). This is the local input for the absorption congruence.
	
	\begin{lemma}\label{lem:reduction-order}
		Let \(\ell\) be a prime and let \(p\notin S_{E,P}\). If \(p\mid D_\ell\), then the reduction \(\widetilde P\in E(\mathbb F_p)\) has exact order \(\ell\).
	\end{lemma}
	
	\begin{proof}
		Since \(p\notin S_{E,P}\), the point \(\widetilde P\) is well-defined and nonzero in \(E(\mathbb F_p)\).
		The condition \(p\mid D_\ell\) is equivalent, by Theorem~\ref{thm:EDS-valuation-law}(1), to
		\[
		r_p\mid \ell,
		\]
		where \(r_p=\operatorname{ord}(\widetilde P)\).
		Because \(\ell\) is prime and \(\widetilde P\ne O\), the order \(r_p\) cannot be \(1\).
		Hence \(r_p=\ell\).
	\end{proof}
	
	\begin{lemma}\label{lem:gcd}
		Let \(\rho\) be a prime number, and let \(\ell_1\) and \(\ell_2\) be distinct rational primes. Then
		\[
		\gcd\bigl(
		\operatorname{rad}_{S_{E,P},\rho}(D_{\ell_1}),
		\operatorname{rad}_{S_{E,P},\rho}(D_{\ell_2})
		\bigr)=1.
		\]
	\end{lemma}
	
	\begin{proof}
		Suppose that a rational prime \(q\) divides both \(\operatorname{rad}_{S_{E,P},\rho}(D_{\ell_1})\) and \(\operatorname{rad}_{S_{E,P},\rho}(D_{\ell_2})\). Then \(q\notin S_{E,P}\), and
		\[
		q\mid D_{\ell_1},
		\qquad
		q\mid D_{\ell_2}.
		\]
		By the preceding lemma, the reduction of \(P\) modulo \(q\) has exact order \(\ell_1\). Applying the same lemma to \(D_{\ell_2}\), the same reduction also has exact order \(\ell_2\). Hence \(\ell_1=\ell_2\), contradicting the assumption that the two primes are distinct. Therefore no such \(q\) exists.
	\end{proof}
	
	For a tuple \(\mathbf n=(n_1,\ldots,n_k)\) of positive integers and a prime \(\ell\), put
	\[
	I_\ell(\mathbf n):=\{i\in\{1,\ldots,k\}:\ell\mid n_i\}.
	\]
	
	\begin{theorem}\label{thm:prime-index-absorption}
		Let \(\rho\) be a prime number.  Let \(n_1,\ldots,n_k\ge 1\), and suppose that
		\[
		\prod_{i=1}^k D_{n_i}\in(\mathbb Q^\times)^\rho.
		\]
		Let \(\ell\) be a rational prime, and let \(p\notin S_{E,P}\) be a prime such that \(p\mid D_\ell\).  Then
		\[
		|I_\ell(\mathbf n)|v_p(D_\ell)
		+
		\sum_{i\in I_\ell(\mathbf n)}v_p(n_i/\ell)
		\equiv 0 \pmod \rho .
		\]
	\end{theorem}
	
	\begin{proof}
		Since \(p\notin S_{E,P}\) and \(p\mid D_\ell\), the preceding lemma gives
		\[
		r_p=\ell.
		\]
		If \(i\notin I_\ell(\mathbf n)\), then \(\ell\nmid n_i\), hence \(r_p\nmid n_i\).
		By the valuation law,
		\[
		v_p(D_{n_i})=0.
		\]
		If \(i\in I_\ell(\mathbf n)\), then \(\ell\mid n_i\), so \(r_p\mid n_i\).
		Again by the valuation law,
		\[
		v_p(D_{n_i})
		=
		v_p(D_{r_p})+v_p(n_i/r_p)
		=
		v_p(D_\ell)+v_p(n_i/\ell).
		\]
		Since
		\[
		\prod_{i=1}^k D_{n_i}\in(\mathbb Q^\times)^\rho,
		\]
		its \(p\)-adic valuation is divisible by \(\rho\). Therefore
		\[
		0
		\equiv
		v_p\left(\prod_{i=1}^kD_{n_i}\right)
		=
		\sum_{i=1}^k v_p(D_{n_i})
		\pmod\rho.
		\]
		Using the preceding two computations, this becomes
		\[
		0
		\equiv
		\sum_{i\in I_\ell(\mathbf n)}
		\bigl(v_p(D_\ell)+v_p(n_i/\ell)\bigr)
		\pmod\rho.
		\]
		Hence
		\[
		|I_\ell(\mathbf n)|v_p(D_\ell)
		+
		\sum_{i\in I_\ell(\mathbf n)}v_p(n_i/\ell)
		\equiv 0 \pmod\rho.
		\]
	\end{proof}
	
	\subsection{Incidence vectors and multiplicity obstruction}
	
	Let
	\[
	V_\rho:=\mathbb F_\rho^k
	\]
	and let \(\mathbf 1=(1,\ldots,1)\in V_\rho\). For a tuple
	\[
	\mathbf n=(n_1,\ldots,n_k)
	\]
	and a prime \(\ell\), define the incidence vector
	\[
	\mathbf e_\ell(\mathbf n)
	:=
	(\delta_{\ell\mid n_1},\ldots,\delta_{\ell\mid n_k})
	\in V_\rho.
	\]
	Thus
	\[
	\langle \mathbf e_\ell(\mathbf n),\mathbf 1\rangle
	\equiv |I_\ell(\mathbf n)|\pmod{\rho}.
	\]
	For a rational prime \(q\), define the valuation vector
	\[
	\mathbf v_q(\mathbf n)
	:=
	(v_q(n_1),\ldots,v_q(n_k))
	\in V_\rho.
	\]
	
	\begin{proposition}\label{prop:incidence-vector}
		Let \(\rho\) be a prime number. Let \(n_1,\ldots,n_k\geq 1\), and suppose that
		\[
		\prod_{i=1}^kD_{n_i}\in(\mathbb Q^\times)^\rho.
		\]
		Let \(\ell\) be a rational prime, and let \(q\notin S_{E,P}\) be a prime such that \(q\mid D_\ell\).
		Then
		\[
		\langle \mathbf e_\ell(\mathbf n),\mathbf v_q(\mathbf n)\rangle
		\equiv
		|I_\ell(\mathbf n)|\bigl(v_q(\ell)-v_q(D_\ell)\bigr)
		\pmod{\rho}.
		\]
		In particular, if \(q\neq \ell\), then
		\[
		\langle \mathbf e_\ell(\mathbf n),\mathbf v_q(\mathbf n)\rangle
		\equiv
		-|I_\ell(\mathbf n)|v_q(D_\ell)
		\pmod{\rho}.
		\]
	\end{proposition}
	
	\begin{proof}
		By Theorem~\ref{thm:prime-index-absorption}, we have
		\[
		|I_\ell(\mathbf n)|v_q(D_\ell)
		+
		\sum_{i\in I_\ell(\mathbf n)}v_q(n_i/\ell)
		\equiv0\pmod{\rho}.
		\]
		On the other hand,
		\[
		\sum_{i\in I_\ell(\mathbf n)}v_q(n_i/\ell)
		=
		\sum_{i\in I_\ell(\mathbf n)}v_q(n_i)
		-
		|I_\ell(\mathbf n)|v_q(\ell).
		\]
		Modulo \(\rho\), the first term is precisely
		\[
		\langle \mathbf e_\ell(\mathbf n),\mathbf v_q(\mathbf n)\rangle.
		\]
		Therefore
		\[
		|I_\ell(\mathbf n)|v_q(D_\ell)
		+
		\langle \mathbf e_\ell(\mathbf n),\mathbf v_q(\mathbf n)\rangle
		-
		|I_\ell(\mathbf n)|v_q(\ell)
		\equiv0\pmod{\rho}.
		\]
		Rearranging gives
		\[
		\langle \mathbf e_\ell(\mathbf n),\mathbf v_q(\mathbf n)\rangle
		\equiv
		|I_\ell(\mathbf n)|\bigl(v_q(\ell)-v_q(D_\ell)\bigr)
		\pmod{\rho}.
		\]
		If \(q\neq \ell\), then \(v_q(\ell)=0\), giving the final assertion.
	\end{proof}
	
	\begin{corollary}\label{cor:multiplicity-obstruction}
		Assume that
		\[
		\prod_{i=1}^kD_{n_i}\in(\mathbb Q^\times)^\rho.
		\]
		Let \(\ell\) be a rational prime, and let
		\[
		q\mid \operatorname{rad}_{S_{E,P},\rho}(D_\ell)
		\]
		with \(q\neq \ell\). Then
		\[
		v_q\left(\prod_{i\in I_\ell(\mathbf n)} n_i\right)
		\equiv
		-|I_\ell(\mathbf n)|v_q(D_\ell)
		\pmod{\rho}.
		\]
		Equivalently,
		\[
		v_q\left(\prod_{i\in I_\ell(\mathbf n)} \frac{n_i}{\ell}\right)
		\equiv
		-|I_\ell(\mathbf n)|v_q(D_\ell)
		\pmod{\rho}.
		\]
		
		In particular:
		
		\begin{enumerate}
			\item If \(\rho\nmid |I_\ell(\mathbf n)|\), then
			\[
			v_q\left(\prod_{i\in I_\ell(\mathbf n)} \frac{n_i}{\ell}\right)
			\not\equiv0\pmod{\rho}.
			\]
			Thus \(q\) occurs in the quotient indices with nonzero multiplicity modulo \(\rho\).
			
			\item If \(\rho\mid |I_\ell(\mathbf n)|\), then
			\[
			v_q\left(\prod_{i\in I_\ell(\mathbf n)} \frac{n_i}{\ell}\right)
			\equiv0\pmod{\rho}.
			\]
			Thus any occurrence of \(q\) among the quotient indices is \(\rho\)-balanced.
		\end{enumerate}
	\end{corollary}
	
	\begin{proof}
		Since \(q\mid \operatorname{rad}_{S_{E,P},\rho}(D_\ell)\), we have
		\[
		q\notin S_{E,P},\qquad q\mid D_\ell,\qquad \rho\nmid v_q(D_\ell).
		\]
		The assertion follows from Proposition~\ref{prop:incidence-vector}, since \(q\neq\ell\).
	\end{proof}
	
	The squarefree case gives a particularly transparent form of the incidence obstruction.
	
	\begin{corollary}\label{cor:squarefree-incidence}
		Assume that \(n_1,\ldots,n_k\) are squarefree, and suppose that
		\[
		\prod_{i=1}^kD_{n_i}\in(\mathbb Q^\times)^\rho.
		\]
		For primes \(\ell\) and \(q\), put
		\[
		N_\ell:=\#\{i:\ell\mid n_i\},
		\qquad
		N_{\ell,q}:=\#\{i:\ell q\mid n_i\}.
		\]
		Let \(q\neq \ell\), and assume that
		\[
		q\mid \operatorname{rad}_{S_{E,P},\rho}(D_\ell).
		\]
		Then
		\[
		N_{\ell,q}
		\equiv
		-N_\ell v_q(D_\ell)
		\pmod{\rho}.
		\]
		In particular, if \(\rho\nmid N_\ell\), then
		\[
		N_{\ell,q}\not\equiv0\pmod{\rho}.
		\]
		If \(\rho\mid N_\ell\), then
		\[
		N_{\ell,q}\equiv0\pmod{\rho}.
		\]
	\end{corollary}
	
	\begin{proof}
		Since the \(n_i\) are squarefree and \(q\neq\ell\), one has
		\[
		v_q\left(\prod_{i\in I_\ell(\mathbf n)} n_i\right)
		=
		\#\{i:\ell q\mid n_i\}
		=
		N_{\ell,q}.
		\]
		The claim follows from Corollary~\ref{cor:multiplicity-obstruction}.
	\end{proof}

	\begin{proposition}\label{prop:prime-index-divisibility}
		Let \(\rho\) be a prime number. Let \(n_1,\ldots,n_k\ge 1\), and suppose that
		\[
		\prod_{i=1}^k D_{n_i}\in(\mathbb Q^\times)^\rho.
		\]
		Let \(\ell\) be a prime such that
		\[
		\rho\nmid |I_\ell(\mathbf n)|.
		\]
		Then
		\[
		\operatorname{rad}_{S_{E,P},\rho}(D_\ell)
		\mid
		\prod_{i\in I_\ell(\mathbf n)}\frac{n_i}{\ell}.
		\]
		Moreover, every prime divisor \(q\) of \(\operatorname{rad}_{S_{E,P},\rho}(D_\ell)\) satisfies
		\[
		\ell\le q+1+2\sqrt q.
		\]
		If, in addition, every \(i\in I_\ell(\mathbf n)\) satisfies \(v_\ell(n_i)=1\) and has \(\ell\) as the largest prime divisor of \(n_i\), then every prime divisor \(q\) of \(\operatorname{rad}_{S_{E,P},\rho}(D_\ell)\) satisfies
		\[
		(\sqrt\ell-1)^2\le q<\ell.
		\]
	\end{proposition}
	
	\begin{proof}
		Let \(q\) be a prime divisor of
		\(\operatorname{rad}_{S_{E,P},\rho}(D_\ell)\). Then \(q\notin S_{E,P}\) and
		\[
		\rho\nmid v_q(D_\ell).
		\]
		In particular \(q\mid D_\ell\). By Theorem~\ref{thm:prime-index-absorption}, applied with \(p=q\), we have
		\[
		|I_\ell(\mathbf n)|v_q(D_\ell)
		+
		\sum_{i\in I_\ell(\mathbf n)} v_q(n_i/\ell)
		\equiv 0 \pmod{\rho}.
		\]
		Since \(\rho\nmid |I_\ell(\mathbf n)|\) and \(\rho\nmid v_q(D_\ell)\), the product
		\[
		|I_\ell(\mathbf n)|v_q(D_\ell)
		\]
		is nonzero modulo \(\rho\). Therefore
		\[
		\sum_{i\in I_\ell(\mathbf n)}v_q(n_i/\ell)
		\not\equiv 0\pmod\rho.
		\]
		In particular this sum is positive, so
		\[
		q\mid \prod_{i\in I_\ell(\mathbf n)}\frac{n_i}{\ell}.
		\]
		Since this holds for every prime divisor \(q\) of \(\operatorname{rad}_{S_{E,P},\rho}(D_\ell)\), the divisibility follows.
		
		Now let \(q\mid \operatorname{rad}_{S_{E,P},\rho}(D_\ell)\). Then \(q\notin S_{E,P}\) and \(q\mid D_\ell\).
		By the lemma, the reduction of \(P\) modulo \(q\) has exact order \(\ell\).
		Hence
		\[
		\ell\mid \#E(\mathbb F_q).
		\]
		The Hasse bound gives
		\[
		\#E(\mathbb F_q)\le q+1+2\sqrt q,
		\]
		and therefore
		\[
		\ell\le q+1+2\sqrt q.
		\]
		
		Finally assume that every \(i\in I_\ell(\mathbf n)\) satisfies \(v_\ell(n_i)=1\) and has \(\ell\) as the largest prime divisor of \(n_i\).  Since \(q\) is prime and
		\[
		q\mid \prod_{i\in I_\ell(\mathbf n)}\frac{n_i}{\ell},
		\]
		there exists \(i\in I_\ell(\mathbf n)\) such that \(q\mid n_i/\ell\).
		Since \(\ell\) is the largest prime divisor of \(n_i\) and \(v_\ell(n_i)=1\), every prime divisor of \(n_i/\ell\) is strictly smaller than \(\ell\).
		Hence \(q<\ell\).
		Combining this with
		\[
		\ell\le q+1+2\sqrt q=(\sqrt q+1)^2
		\]
		gives
		\[
		q\ge(\sqrt\ell-1)^2.
		\]
		Thus
		\[
		(\sqrt\ell-1)^2\le q<\ell.
		\]
	\end{proof}

	\section{A finiteness input for power classes at prime indices}
	
	We first isolate the precise part of Reynolds' argument which is needed below. In Reynolds' notation, our sequence \(D_m\) is the sequence denoted \(B_m\).
	
	\begin{proposition}[A fixed-\(S\) finiteness consequence]
		\label{prop:reynolds-fixed-S}
		Assume that \(D_1\) is divisible by \(2\) or \(3\).
		Let \(S\) be a fixed finite set of rational primes, and let \(\rho\) be a fixed prime number.
		Then there are only finitely many indices \(m\geq 1\) for which
		\[
		D_m\in (\mathbb Q_{>0}^{\times})^\rho\cdot
		\langle p:p\in S\rangle .
		\]
		Equivalently, after absorbing \(\rho\)-th powers supported on \(S\), there are only finitely many \(m\) for which
		\[
		D_m=u y^\rho,\qquad
		u=\prod_{p\in S}p^{r_p},\quad 0\le r_p<\rho,\quad y\in\mathbb Z_{>0}.
		\]
	\end{proposition}
	
	\begin{proof}
		We spell out the fixed-\(S\), fixed-exponent consequence of Reynolds' argument which is used in the sequel. In Reynolds' notation the
		elliptic divisibility sequence is denoted \((B_m)\), which is our sequence \((D_m)\). Reynolds proves finiteness of perfect-power terms under the hypothesis that \(D_1\) is divisible by \(2\) or \(3\), and observes in \cite[Remark 4.2]{Reynolds2012} that the same argument remains valid after adjoining any prescribed finite set of primes to the auxiliary exceptional set. Equivalently, for each fixed finite set \(S\), each fixed exponent \(\rho\), and each fixed \(S\)-unit coefficient \(u\), the equation
		\[
		D_m=u y^\rho
		\]
		has only finitely many solutions \(m\ge 1\).
		
		Let
		\[
		U_{S,\rho}
		:=
		\left\{
		\prod_{p\in S}p^{r_p}:0\leq r_p<\rho
		\right\}.
		\]
		This is a finite set. The condition
		\[
		D_m\in (\mathbb Q^\times_{>0})^\rho\cdot \langle p:p\in S\rangle
		\]
		is equivalent, after absorbing \(\rho\)-th powers supported on \(S\), to
		\[
		D_m=u y^\rho
		\]
		for some \(u\in U_{S,\rho}\) and \(y\in\mathbb Z_{>0}\). For each fixed \(u\in U_{S,\rho}\), Reynolds' fixed-\(S\) form gives finiteness of the indices \(m\) satisfying this equation. Since \(U_{S,\rho}\) is finite, the union over \(u\in U_{S,\rho}\) is finite. This proves the assertion.
	\end{proof}
	
	\begin{theorem}
		\label{thm:fixed-S-radical}
		Assume that \(D_1\) is divisible by \(2\) or by \(3\). Let \(S\) be a fixed finite set of rational primes, and let \(\rho\) be a fixed prime number.
		Then there are only finitely many indices \(m\geq 1\) for which
		\[
		\operatorname{rad}_{S,\rho}(D_m)=1.
		\]
	\end{theorem}
	
	\begin{proof}
		By definition,
		\[
		\operatorname{rad}_{S,\rho}(D_m)=1
		\]
		if and only if
		\[
		\rho\mid v_q(D_m)
		\qquad\text{for every prime }q\notin S.
		\]
		Equivalently,
		\[
		D_m\in(\mathbb Q_{>0}^{\times})^\rho
		\cdot \langle p:p\in S\rangle .
		\]
		The assertion follows from the preceding proposition.
	\end{proof}

	\begin{remark}
		The only global finiteness input used in the sequel is the fixed-\(S\) consequence of Reynolds' theorem recorded in Proposition~\ref{prop:reynolds-fixed-S}, or equivalently the radical formulation in Theorem~\ref{thm:fixed-S-radical}. All subsequent arguments require only the following consequence: for each fixed prime \(\rho\), all but finitely many prime indices
		\(\ell\) satisfy
		\[
		\operatorname{rad}_{S_{E,P},\rho}(D_\ell)>1.
		\]
		Thus the use of Reynolds' theorem is confined to producing a \(\rho\)-detecting prime divisor at sufficiently large prime indices. The absorption, incidence, clustering, and large-prime-gap arguments are otherwise local consequences of the valuation law and the Hasse bound.
	\end{remark}
	
	We isolate this consequence as a separate input.
	
	\begin{hypothesis}[\(\rho\)-detecting prime-index input]\label{hyp:detecting}
		Let \(\rho\) be a fixed prime. There exists a constant \(L_\rho=L_\rho(E,P)\) such that, for every prime \(\ell>L_\rho\),
		\[
		\operatorname{rad}_{S_{E,P},\rho}(D_\ell)>1.
		\]
		Equivalently, for every prime \(\ell>L_\rho\), there exists a prime \(p\notin S_{E,P}\) such that
		\[
		p\mid D_\ell,\qquad \rho\nmid v_p(D_\ell).
		\]
	\end{hypothesis}
	
	\begin{proposition}[Reynolds' theorem supplies Hypothesis \ref{hyp:detecting}]\label{prop:large-prime-index-detecting-radical}
		Assume that \(D_1\) is divisible by \(2\) or by \(3\). Then, for every fixed prime \(\rho\), Hypothesis \ref{hyp:detecting} holds.
	\end{proposition}
	
	\begin{proof}
		Apply Theorem~\ref{thm:fixed-S-radical} with \(S=S_{E,P}\). It gives only finitely many indices \(m\geq 1\) for which
		\[
		\operatorname{rad}_{S_{E,P},\rho}(D_m)=1.
		\]
		In particular, only finitely many prime indices \(\ell\) have this property.
		Choose \(L_\rho\) larger than all such exceptional prime indices.
		
		If \(\ell>L_\rho\) is prime, then
		\[
		\operatorname{rad}_{S_{E,P},\rho}(D_\ell)>1.
		\]
		This is equivalent to the existence of a prime \(p\notin S_{E,P}\) such that
		\[
		p\mid D_\ell,\qquad \rho\nmid v_p(D_\ell).
		\]
		Thus Hypothesis \ref{hyp:detecting} holds.
	\end{proof}
	
	We call a prime \(p\) a primitive prime divisor of \(D_m\) if \(p\mid D_m\) and \(p\nmid D_j\) for every \(1\leq j<m\).
	
	\begin{theorem}
		\label{thm:rho-detecting-primitive-divisor}
		Assume that \(D_1\) is divisible by \(2\) or \(3\). Let \(\rho\) be a fixed prime number.
		Then there exists a constant \(L_\rho=L_\rho(E,P)\) such that, for every prime \(\ell>L_\rho\), there exists a primitive prime divisor \(p\mid D_\ell\) such that
		\[
		p\notin S_{E,P},\qquad \rho\nmid v_p(D_\ell).
		\]
	\end{theorem}
	
	\begin{proof}
		Let \(L_\rho\) be the constant supplied by Proposition~\ref{prop:large-prime-index-detecting-radical}.
		If \(\ell>L_\rho\) is prime, then the proposition gives a prime \(p\notin S_{E,P}\) such that
		\[
		p\mid D_\ell,\qquad \rho\nmid v_p(D_\ell).
		\]
		It remains to prove that \(p\) is primitive for \(D_\ell\).
		
		Since \(p\notin S_{E,P}\) and \(p\mid D_\ell\), the valuation law gives \(r_p\mid \ell\), where \(r_p\) is the order of the reduction of \(P\)
		modulo \(p\). Because \(\ell\) is prime and the reduction of \(P\) is nonzero outside \(S_{E,P}\), one has \(r_p\ne 1\). Hence
		\[
		r_p=\ell.
		\]
		If \(1\le m<\ell\), then \(r_p\nmid m\), and the valuation law gives
		\[
		v_p(D_m)=0.
		\]
		Thus \(p\nmid D_m\) for every \(1\le m<\ell\). Therefore \(p\) is a primitive prime divisor of \(D_\ell\), and by construction
		\[
		p\notin S_{E,P},\qquad \rho\nmid v_p(D_\ell).
		\]
		This proves the theorem.
	\end{proof}

	\section{Prime-support obstructions and top-prime clustering}
	We now assemble the argument under the detecting-prime input of Hypothesis~\ref{hyp:detecting}.
	The local absorption congruence, combined with this input and the Hasse bound, yields global constraints on the prime support of the indices. The point of this section is threefold.
	First, a single large prime index \(\ell\) forces the \(\rho\)-power-free part of \(D_\ell\) to be absorbed by the quotient indices.
	Second, distinct prime indices give pairwise coprime detecting radicals.
	Third, under the top-prime hypotheses, the corresponding incidence supports among the indices are disjoint, while the product relation forces them to be \(\rho\)-balanced.
	
	The obstruction obtained below is stronger than ordinary support divisibility.
	A prime divisor of \(D_\ell\) is forced to occur in the restricted quotient
	\[
	\prod_{i\in I_\ell(\mathbf n)}\frac{n_i}{\ell},
	\]
	and, when \(\rho\nmid |I_\ell(\mathbf n)|\), it occurs there with nonzero total multiplicity modulo \(\rho\).
	For distinct large prime indices, the corresponding detecting supports are disjoint.
	Thus the obstruction accumulates multiplicatively across the large prime divisors of the indices.

	For an integer \(N\geq 2\), let
	\[
	P^+(N):=\max\{p:p\mid N\}
	\]
	denote the largest prime divisor of \(N\). We also put \(P^+(1)=1\).
	
	For a real number \(B\ge 2\), a positive integer \(N\) is called \(B\)-smooth if \(P^+(N)\le B\).
	
	The next result is the effective form of Proposition~\ref{prop:prime-index-divisibility} at all sufficiently large prime indices, where the nontriviality of the detecting radical is supplied by Hypothesis~\ref{hyp:detecting}.	
	
	\begin{proposition}\label{prop:single-prime-support}
		Assume Hypothesis~\ref{hyp:detecting} for the prime \(\rho\), and let \(L_\rho\) be the constant appearing there. Let \(n_1,\ldots,n_k\ge 1\) satisfy
		\[
		\prod_{i=1}^k D_{n_i}\in(\mathbb Q^\times)^\rho.
		\]
		Let \(\ell>L_\rho\) be a prime such that
		\[
		\rho\nmid |I_\ell(\mathbf n)|.
		\]
		Then
		\[
		\operatorname{rad}_{S_{E,P},\rho}(D_\ell)>1
		\]
		and
		\[
		\operatorname{rad}_{S_{E,P},\rho}(D_\ell)
		\mid
		\prod_{i\in I_\ell(\mathbf n)}\frac{n_i}{\ell}.
		\]
		Moreover, every prime divisor \(q\) of \(\operatorname{rad}_{S_{E,P},\rho}(D_\ell)\) satisfies
		\[
		\ell\le q+1+2\sqrt q.
		\]
		If, in addition, every \(i\in I_\ell(\mathbf n)\) satisfies \(v_\ell(n_i)=1\) and has \(\ell\) as the largest prime divisor of \(n_i\), then every prime divisor \(q\) of \(\operatorname{rad}_{S_{E,P},\rho}(D_\ell)\) satisfies
		\[
		(\sqrt\ell-1)^2\le q<\ell.
		\]
	\end{proposition}
	
	\begin{proof}
		Since \(\ell>L_\rho\), Hypothesis~\ref{hyp:detecting} gives
		\[
		\operatorname{rad}_{S_{E,P},\rho}(D_\ell)>1.
		\]
		The remaining assertions are precisely Proposition~\ref{prop:prime-index-divisibility} applied to the present tuple
		\((n_1,\ldots,n_k)\).
	\end{proof}
	
	\begin{proposition}\label{thm:simultaneous-prime-support}
		Assume Hypothesis~\ref{hyp:detecting} for the prime \(\rho\), and let \(L_\rho=L_\rho(E,P)\) be as in Proposition~\ref{prop:single-prime-support}. Let
		\(n_1,\ldots,n_k\ge 1\) satisfy
		\[
		\prod_{i=1}^k D_{n_i}\in(\mathbb Q^\times)^\rho.
		\]
		Let \(\Lambda\) be a finite set of primes \(\ell>L_\rho\) such that
		\[
		\rho\nmid |I_\ell(\mathbf n)|
		\]
		for every \(\ell\in\Lambda\). Then the integers
		\[
		\operatorname{rad}_{S_{E,P},\rho}(D_\ell),\qquad \ell\in\Lambda,
		\]
		are pairwise coprime, and
		\[
		\prod_{\ell\in\Lambda}
		\operatorname{rad}_{S_{E,P},\rho}(D_\ell)
		\mid
		\prod_{\ell\in\Lambda}
		\prod_{i\in I_\ell(\mathbf n)}
		\frac{n_i}{\ell}.
		\]
		Moreover, if \(q\) is a prime divisor of \(\operatorname{rad}_{S_{E,P},\rho}(D_\ell)\) for some \(\ell\in\Lambda\), then
		\[
		\ell\le q+1+2\sqrt q.
		\]
		If, in addition, for every \(\ell\in\Lambda\) and every \(i\in I_\ell(\mathbf n)\), one has
		\[
		v_\ell(n_i)=1,\qquad \ell=P^+(n_i),
		\]
		then every such prime divisor \(q\) satisfies
		\[
		(\sqrt\ell-1)^2\le q<\ell.
		\]
	\end{proposition}
	
	\begin{proof}
		For each \(\ell\in\Lambda\), Proposition~\ref{prop:single-prime-support} gives
		\[
		\operatorname{rad}_{S_{E,P},\rho}(D_\ell)>1
		\]
		and
		\[
		\operatorname{rad}_{S_{E,P},\rho}(D_\ell)
		\mid
		\prod_{i\in I_\ell(\mathbf n)}\frac{n_i}{\ell}.
		\]
		The pairwise coprimality of the integers \(\operatorname{rad}_{S_{E,P},\rho}(D_\ell)\), with \(\ell\in\Lambda\), follows from the disjointness lemma in Section~3. Multiplying the displayed divisibilities over all \(\ell\in\Lambda\) gives
		\[
		\prod_{\ell\in\Lambda}
		\operatorname{rad}_{S_{E,P},\rho}(D_\ell)
		\mid
		\prod_{\ell\in\Lambda}
		\prod_{i\in I_\ell(\mathbf n)}\frac{n_i}{\ell}.
		\]
		The Hasse-scale bound
		\[
		\ell\leq q+1+2\sqrt q
		\]
		for each prime divisor \(q\) of \(\operatorname{rad}_{S_{E,P},\rho}(D_\ell)\) is also part of Proposition~\ref{prop:single-prime-support}.
		The final assertion follows from the final assertion of Proposition~\ref{prop:single-prime-support} applied to each \(\ell\in\Lambda\).
	\end{proof}
	
	\begin{corollary}\label{cor:radical-lower-bound}
		Assume the hypotheses of Proposition~\ref{thm:simultaneous-prime-support}. Suppose moreover that, for every \(\ell\in\Lambda\) and every \(i\in I_\ell(\mathbf n)\), one has
		\[
		v_\ell(n_i)=1,
		\qquad
		\ell=P^+(n_i).
		\]
		Then
		\[
		\operatorname{rad}\left(
		\prod_{\ell\in\Lambda}
		\prod_{i\in I_\ell(\mathbf n)}\frac{n_i}{\ell}
		\right)
		\geq
		\prod_{\ell\in\Lambda}(\sqrt{\ell}-1)^2 .
		\]
	\end{corollary}
	
	\begin{proof}
		For each \(\ell\in\Lambda\), Proposition~\ref{prop:single-prime-support} gives
		\[
		\operatorname{rad}_{S_{E,P},\rho}(D_\ell)>1.
		\]
		Choose a prime divisor
		\[
		q_\ell\mid \operatorname{rad}_{S_{E,P},\rho}(D_\ell).
		\]
		By the disjointness lemma in Section~3, the primes \(q_\ell\), as \(\ell\) varies in \(\Lambda\), are pairwise distinct. Proposition~\ref{prop:single-prime-support} shows that each \(q_\ell\) divides
		\[
		\prod_{\ell'\in\Lambda}
		\prod_{i\in I_{\ell'}(\mathbf n)}\frac{n_i}{\ell'}
		\]
		and, under the top-prime hypothesis,
		\[
		q_\ell\geq(\sqrt{\ell}-1)^2 .
		\]
		Therefore the radical of the displayed product is at least
		\[
		\prod_{\ell\in\Lambda} q_\ell
		\geq
		\prod_{\ell\in\Lambda}(\sqrt{\ell}-1)^2 .
		\]
	\end{proof}
	
	This lower bound is the quantitative form of the fact that distinct large prime indices require distinct absorbing primes in the quotient indices.

	Before deriving the smooth-cofactor obstruction, we record the relevant multiplicity consequence.
	Under the hypotheses of Proposition~\ref{prop:single-prime-support} and the top-prime condition \(v_\ell(n_i)=1\), \(\ell=P^+(n_i)\) for \(i\in I_\ell(\mathbf n)\), every prime
	\[
	q\mid \operatorname{rad}_{S_{E,P},\rho}(D_\ell)
	\]
	lies in the interval
	\[
	(\sqrt{\ell}-1)^2\le q<\ell
	\]
	and occurs in
	\[
	\prod_{i\in I_\ell(\mathbf n)}\frac{n_i}{\ell}
	\]
	with nonzero total \(q\)-adic multiplicity modulo \(\rho\).
	This is the combination of Proposition~\ref{prop:single-prime-support} with the multiplicity congruence of Corollary~\ref{cor:multiplicity-obstruction}.
	
	\begin{proposition}\label{thm:smooth-cofactor-obstruction}
		Assume Hypothesis \ref{hyp:detecting} for the prime \(\rho\), and let \(L_\rho\) be the constant appearing there. Let \(B\geq 2\), and suppose that
		\[
		\prod_{i=1}^k D_{n_i}\in(\mathbb Q^\times)^\rho.
		\]
		Let \(\ell\) be a prime satisfying
		\[
		\ell>\max\{L_\rho,(\sqrt B+1)^2\}.
		\]
		Assume that, for every \(i\in I_\ell(\mathbf n)\), one has
		\[
		v_\ell(n_i)=1,
		\qquad
		\ell=P^+(n_i),
		\]
		and that \(n_i/\ell\) is \(B\)-smooth. Then
		\[
		\rho\mid |I_\ell(\mathbf n)|.
		\]
	\end{proposition}
	
	\begin{proof}
		Suppose, for contradiction, that
		\[
		\rho\nmid |I_\ell(\mathbf n)|.
		\]
		Since \(\ell>L_\rho\), Proposition~\ref{prop:single-prime-support} applies and gives
		\[
		\operatorname{rad}_{S_{E,P},\rho}(D_\ell)>1
		\]
		and
		\[
		\operatorname{rad}_{S_{E,P},\rho}(D_\ell)
		\mid
		\prod_{i\in I_\ell(\mathbf n)}\frac{n_i}{\ell}.
		\]
		Choose a prime divisor
		\[
		q\mid \operatorname{rad}_{S_{E,P},\rho}(D_\ell).
		\]
		By the top-prime hypothesis and Proposition~\ref{prop:single-prime-support}, one has
		\[
		(\sqrt{\ell}-1)^2\leq q<\ell.
		\]
		On the other hand, each \(n_i/\ell\), with \(i\in I_\ell(\mathbf n)\), is \(B\)-smooth. Hence the product
		\[
		\prod_{i\in I_\ell(\mathbf n)}\frac{n_i}{\ell}
		\]
		is also \(B\)-smooth. Since \(q\) divides this product, we have \(q\leq B\).
		Therefore
		\[
		(\sqrt{\ell}-1)^2\leq B,
		\]
		which is equivalent to
		\[
		\ell\leq(\sqrt B+1)^2.
		\]
		This contradicts the hypothesis
		\[
		\ell>(\sqrt B+1)^2.
		\]
		Hence \(\rho\mid |I_\ell(\mathbf n)|\).
	\end{proof}
	
	We now pass from the one-prime balancing obstruction to a genuinely \(k\)-term consequence. Under the top-prime hypotheses, the incidence
	supports attached to distinct large top primes are automatically disjoint. The arithmetic content is that the product relation forces each nonempty such support to be \(\rho\)-balanced. It is useful to regard these large top prime divisors as defining an incidence matrix over \(\mathbb F_\rho\).

	\begin{theorem}[Large top-prime cluster-packing obstruction]\label{thm:large-top-prime-cluster-packing}
		Assume Hypothesis~\ref{hyp:detecting} for the prime \(\rho\), and let \(L_\rho\) be the constant appearing there.
		Let \(B\geq 2\), and put
		\[
		C_{\rho,B}:=\max\{L_\rho,(\sqrt B+1)^2\}.
		\]
		Let \(n_1,\ldots,n_k\geq 1\) satisfy
		\[
		\prod_{i=1}^k D_{n_i}\in (\mathbb Q^\times)^\rho .
		\]
		Let \(\Lambda\) be a finite set of primes \(\ell>C_{\rho,B}\) such that, for every \(\ell\in\Lambda\) and every \(i\in I_\ell(\mathbf n)\), one has
		\[
		v_\ell(n_i)=1,\qquad \ell=P^+(n_i),
		\]
		and \(n_i/\ell\) is \(B\)-smooth.
		
		Define
		\[
		M_\Lambda(\mathbf n):=(\delta_{\ell\mid n_i})_{\ell\in\Lambda,\ 1\leq i\leq k}
		\in \operatorname{Mat}_{|\Lambda|\times k}(\mathbb F_\rho),
		\]
		and put
		\[
		\Lambda^\ast:=\{\ell\in\Lambda:I_\ell(\mathbf n)\neq\varnothing\}.
		\]
		Then the following hold.
		
		\begin{enumerate}
			\item For every \(\ell\in\Lambda\),
			\[
			\rho\mid |I_\ell(\mathbf n)|.
			\]
			Equivalently,
			\[
			M_\Lambda(\mathbf n)
			\begin{pmatrix}
				1\\
				\vdots\\
				1
			\end{pmatrix}
			=0
			\quad\text{in }\mathbb F_\rho^{|\Lambda|}.
			\]
			
			\item If \(\ell,\ell'\in\Lambda^\ast\) and \(\ell\neq \ell'\), then
			\[
			I_\ell(\mathbf n)\cap I_{\ell'}(\mathbf n)=\varnothing .
			\]
			Thus the nonzero rows of \(M_\Lambda(\mathbf n)\) have pairwise disjoint supports.
			
			\item The nonzero rows of \(M_\Lambda(\mathbf n)\) are linearly independent over \(\mathbb F_\rho\). Hence
			\[
			\operatorname{rank}_{\mathbb F_\rho} M_\Lambda(\mathbf n)=|\Lambda^\ast|.
			\]
			
			\item One has the packing bound
			\[
			|\Lambda^\ast|\leq \left\lfloor \frac{k}{\rho}\right\rfloor .
			\]
			In particular, if
			\[
			|\Lambda^\ast|>\left\lfloor \frac{k}{\rho}\right\rfloor,
			\]
			then no relation
			\[
			\prod_{i=1}^k D_{n_i}\in(\mathbb Q^\times)^\rho
			\]
			can hold.
			
			\item If \(k<\rho\), then \(\Lambda^\ast=\varnothing\). Equivalently, no prime \(\ell\in\Lambda\) divides any of the indices \(n_i\).
		\end{enumerate}
	\end{theorem}
	
	\begin{proof}
		Let \(\ell\in\Lambda\). By hypothesis, every index \(n_i\) divisible by \(\ell\) has \(\ell\) as a simple largest prime divisor, and the corresponding cofactor \(n_i/\ell\) is \(B\)-smooth. Since
		\[
		\ell>C_{\rho,B}=\max\{L_\rho,(\sqrt B+1)^2\},
		\]
		Proposition~\ref{thm:smooth-cofactor-obstruction} applies and gives
		\[
		\rho\mid |I_\ell(\mathbf n)|.
		\]
		The \(\ell\)-row of \(M_\Lambda(\mathbf n)\) is precisely the incidence vector
		\[
		e_\ell(\mathbf n)=(\delta_{\ell\mid n_1},\ldots,\delta_{\ell\mid n_k}),
		\]
		whose Hamming weight is \(|I_\ell(\mathbf n)|\). Hence
		\[
		\langle e_\ell(\mathbf n),(1,\ldots,1)\rangle
		=|I_\ell(\mathbf n)|\equiv 0\pmod \rho.
		\]
		This proves
		\[
		M_\Lambda(\mathbf n)(1,\ldots,1)^t=0.
		\]
		
		Now let \(\ell,\ell'\in\Lambda^\ast\) be distinct. Suppose that \(i\in I_\ell(\mathbf n)\cap I_{\ell'}(\mathbf n)\). Then both \(\ell\) and
		\(\ell'\) divide \(n_i\). By the defining hypothesis on \(\Lambda\), since \(i\in I_\ell(\mathbf n)\), one has
		\[
		\ell=P^+(n_i).
		\]
		Similarly, since \(i\in I_{\ell'}(\mathbf n)\), one has
		\[
		\ell'=P^+(n_i).
		\]
		Thus \(\ell=\ell'\), a contradiction. Therefore
		\[
		I_\ell(\mathbf n)\cap I_{\ell'}(\mathbf n)=\varnothing .
		\]
		
		It follows that the nonzero rows of \(M_\Lambda(\mathbf n)\) have pairwise disjoint supports. Such rows are linearly independent over any field: indeed, if
		\[
		\sum_{\ell\in\Lambda^\ast} a_\ell e_\ell(\mathbf n)=0
		\quad\text{in }\mathbb F_\rho^k,
		\]
		then choosing a coordinate in the support of a fixed nonzero row \(e_\ell(\mathbf n)\) gives \(a_\ell=0\). Hence all \(a_\ell=0\), and therefore
		\[
		\operatorname{rank}_{\mathbb F_\rho}M_\Lambda(\mathbf n)=|\Lambda^\ast|.
		\]
		
		Finally, the sets \(I_\ell(\mathbf n)\), \(\ell\in\Lambda^\ast\), are pairwise disjoint subsets of \(\{1,\ldots,k\}\). Moreover each of them has cardinality divisible by \(\rho\) and is nonempty, so
		\[
		|I_\ell(\mathbf n)|\geq \rho
		\]
		for every \(\ell\in\Lambda^\ast\). Therefore
		\[
		k\geq \sum_{\ell\in\Lambda^\ast}|I_\ell(\mathbf n)|
		\geq \rho|\Lambda^\ast|.
		\]
		Hence
		\[
		|\Lambda^\ast|\leq \left\lfloor \frac{k}{\rho}\right\rfloor .
		\]
		The remaining assertions are immediate consequences of this bound.
	\end{proof}
	
	For the rest of this section, let \(L_\rho=L_\rho(E,P)\) be as in Theorem~\ref{thm:rho-detecting-primitive-divisor}.
	
	\begin{corollary}\label{cor:repeated-large-top-prime}
		Assume that \(D_1\) is divisible by \(2\) or \(3\), and let \(\rho\) be a prime number. Let \(B\geq 2\), and put
		\[
		C_{\rho,B}:=\max\{L_\rho,(\sqrt B+1)^2\}.
		\]
		Let
		\[
		n_i=\ell_i a_i\qquad (1\leq i\leq k),
		\]
		where each \(\ell_i>C_{\rho,B}\) is prime and each \(a_i\) is \(B\)-smooth.
		Then
		\[
		\prod_{i=1}^k D_{n_i}\in(\mathbb Q^\times)^\rho
		\]
		implies that, for every prime \(\ell\) occurring among \(\ell_1,\ldots,\ell_k\),
		\[
		\#\{i:\ell_i=\ell\}\equiv 0\pmod \rho.
		\]
		In particular, each large top prime \(\ell_i\) must occur among the indices \(n_1,\ldots,n_k\) with multiplicity at least \(\rho\). Consequently, if the primes \(\ell_1,\ldots,\ell_k\) are pairwise distinct, then
		\[
		\prod_{i=1}^k D_{n_i}\notin(\mathbb Q^\times)^\rho.
		\]
	\end{corollary}
	
	\begin{proof}
		Since
		\[
		\ell_i>C_{\rho,B}\geq(\sqrt B+1)^2>B,
		\]
		and \(a_i\) is \(B\)-smooth, one has
		\[
		v_{\ell_i}(n_i)=1,\qquad \ell_i=P^+(n_i).
		\]
		Moreover, if \(\ell=\ell_i\) for some \(i\), then \(\ell\) cannot divide \(a_j\) for any \(j\), because \(a_j\) is \(B\)-smooth and \(\ell>B\). Hence
		\[
		I_\ell(\mathbf n)=\{i:\ell_i=\ell\}.
		\]
		Applying the preceding theorem to the set of distinct primes occurring among \(\ell_1,\ldots,\ell_k\), we obtain
		\[
		\rho\mid |I_\ell(\mathbf n)|=\#\{i:\ell_i=\ell\}
		\]
		for each such \(\ell\).
	\end{proof}
	
	The coprime two-term case is the extremal situation in which a large top prime appears in exactly one index. In that case the balancing alternative is impossible, and one obtains a large-prime-gap exclusion.
	
	\begin{proposition}\label{thm:large-prime-gap}
		Assume that \(D_1\) is divisible by \(2\) or \(3\), and let \(\rho\) be a prime number. Let \(m\geq 2\) and \(n\geq 1\) be coprime integers.
		Put
		\[
		\ell=P^+(m),
		\]
		and assume that \(v_\ell(m)=1\). 
		If
		\[
		D_mD_n\in(\mathbb Q^\times)^\rho
		\]
		and \(\ell>L_\rho\), then \(m/\ell>1\) and
		\[
		P^+(m/\ell)\ge(\sqrt{\ell}-1)^2
		=\ell-2\sqrt{\ell}+1.
		\]
		Consequently,
		\[
		(\sqrt{\ell}-1)^2\leq P^+(m/\ell)<\ell.
		\]
		
		Equivalently, if
		\[
		P^+(m/\ell)<(\sqrt{\ell}-1)^2,
		\]
		then
		\[
		D_mD_n\notin(\mathbb Q^\times)^\rho
		\]
		for every integer \(n\geq 1\) coprime to \(m\).
	\end{proposition}
	
	\begin{proof}
		Suppose first that
		\[
		D_mD_n\in(\mathbb Q^\times)^\rho
		\]
		with \((m,n)=1\), and let \(\ell=P^+(m)>L_\rho\). Since \(\ell\mid m\) and \((m,n)=1\), we have \(\ell\nmid n\).
		Hence, for the tuple \((m,n)\),
		\[
		I_\ell(m,n)=\{1\}.
		\]
		Thus \(|I_\ell(m,n)|=1\), and in particular
		\[
		\rho\nmid |I_\ell(m,n)|.
		\]
		
		By Proposition~\ref{prop:single-prime-support}, applied to the tuple \((m,n)\), we obtain
		\[
		\operatorname{rad}_{S_{E,P},\rho}(D_\ell)>1
		\]
		and
		\[
		\operatorname{rad}_{S_{E,P},\rho}(D_\ell)\mid \frac{m}{\ell}.
		\]
		In particular, the quotient \(m/\ell\) cannot be equal to \(1\), since \(rad_{S_{E,P},\rho}(D_\ell)>1\).
		Therefore \(m/\ell>1\). Choose a prime divisor
		\[
		q\mid \operatorname{rad}_{S_{E,P},\rho}(D_\ell).
		\]
		Then \(q\mid m/\ell\), and Proposition~\ref{prop:single-prime-support} also gives the Hasse bound constraint
		\[
		\ell\leq q+1+2\sqrt q=(\sqrt q+1)^2.
		\]
		It follows that
		\[
		q\geq(\sqrt{\ell}-1)^2.
		\]
		Since \(q\mid m/\ell\), we have
		\[
		P^+(m/\ell)\geq q\geq(\sqrt{\ell}-1)^2.
		\]
		
		On the other hand, since \(\ell=P^+(m)\) and \(v_\ell(m)=1\), every prime divisor of \(m/\ell\) is strictly smaller than \(\ell\). Hence
		\[
		P^+(m/\ell)<\ell.
		\]
		This proves
		\[
		(\sqrt{\ell}-1)^2\leq P^+(m/\ell)<\ell.
		\]
		
		The final assertion is the contrapositive. If
		\[
		P^+(m/\ell)<(\sqrt{\ell}-1)^2,
		\]
		then the preceding necessary condition fails, and therefore no relation
		\[
		D_mD_n\in(\mathbb Q^\times)^\rho
		\]
		can hold with \((m,n)=1\).
	\end{proof}
	
	\begin{corollary}\label{cor:smooth-cofactor-exclusion}
		Assume that \(D_1\) is divisible by \(2\) or \(3\), and let \(\rho\) be a prime number. Let \(B\geq 2\), and let \(m\geq 2\).
		Put
		\[
		\ell=P^+(m),
		\]
		and assume that \(v_\ell(m)=1\). Suppose that \(m/\ell\) is \(B\)-smooth.
		
		If
		\[
		\ell>L_\rho
		\qquad\text{and}\qquad
		\ell>(\sqrt B+1)^2,
		\]
		then
		\[
		D_mD_n\notin(\mathbb Q^\times)^\rho
		\]
		for every integer \(n\geq 1\) coprime to \(m\).
		
		Equivalently, if there exists \(n\geq 1\) with \((m,n)=1\) such that
		\[
		D_mD_n\in(\mathbb Q^\times)^\rho,
		\]
		then
		\[
		\ell\leq \max\{L_\rho,(\sqrt B+1)^2\}.
		\]
	\end{corollary}
	
	\begin{proof}
		Assume that there exists \(n\geq 1\) with \((m,n)=1\) such that
		\[
		D_mD_n\in(\mathbb Q^\times)^\rho.
		\]
		If \(\ell\leq L_\rho\), there is nothing to prove. Suppose that \(\ell>L_\rho\).
		By Proposition~\ref{thm:large-prime-gap},
		\[
		P^+(m/\ell)\geq(\sqrt{\ell}-1)^2.
		\]
		Since \(m/\ell\) is \(B\)-smooth, we have
		\[
		P^+(m/\ell)\leq B.
		\]
		Therefore
		\[
		(\sqrt{\ell}-1)^2\leq B.
		\]
		Equivalently,
		\[
		\ell\leq(\sqrt B+1)^2.
		\]
		Hence
		\[
		\ell\leq\max\{L_\rho,(\sqrt B+1)^2\}.
		\]
		
		The exclusion formulation is the contrapositive.
	\end{proof}
	
	Taking \(\rho=2\) recovers the square-product version of the preceding obstructions.
	In this case the condition \(\rho\nmid |I_\ell(\mathbf n)|\) means that \(|I_\ell(\mathbf n)|\) is odd, and	\(\operatorname{rad}_{S_{E,P},2}\) is the squarefree part outside \(S_{E,P}\).
	
	\begin{corollary}\label{cor:square-case}
		Assume that \(D_1\) is divisible by \(2\) or \(3\). Then there exists a constant \(L_2=L_2(E,P)\) such that the following holds.
		
		Let \(n_1,\ldots,n_k\geq 1\) satisfy
		\[
		\prod_{i=1}^k D_{n_i}\in (\mathbb Q^\times)^2 .
		\]
		Let \(\ell>L_2\) be a prime such that \(|I_\ell(\mathbf n)|\) is odd.
		Then
		\[
		\operatorname{sqf}_{S_{E,P}}(D_\ell)>1
		\]
		and
		\[
		\operatorname{sqf}_{S_{E,P}}(D_\ell)
		\mid
		\prod_{i\in I_\ell(\mathbf n)}\frac{n_i}{\ell}.
		\]
		Moreover, every prime divisor \(q\) of \(\operatorname{sqf}_{S_{E,P}}(D_\ell)\) satisfies
		\[
		\ell\leq q+1+2\sqrt q.
		\]
		If, in addition, every \(i\in I_\ell(\mathbf n)\) satisfies \(v_\ell(n_i)=1\) and has \(\ell\) as the largest prime divisor of \(n_i\), then every such prime divisor \(q\) satisfies
		\[
		(\sqrt{\ell}-1)^2\leq q<\ell.
		\]
	\end{corollary}
	
	\begin{proof}
		This is Proposition~\ref{prop:single-prime-support} with \(\rho=2\), since
		\[
		\operatorname{rad}_{S_{E,P},2}(D_\ell)
		=
		\operatorname{sqf}_{S_{E,P}}(D_\ell),
		\]
		and the condition \(2\nmid |I_\ell(\mathbf n)|\) is precisely that \(|I_\ell(\mathbf n)|\) is odd.
	\end{proof}

	\section*{Statements and Declarations}
	
	\noindent\textbf{Competing interests.}
	The author declares no competing interests.
	
	\noindent\textbf{Funding.}
	No funding was received for this work.
	
	\noindent\textbf{Data availability.}
	Data sharing is not applicable to this article, as no datasets were generated
	or analysed.


\begin{thebibliography}{99}
		
		\bibitem{Alfaraj2023}
		A. Alfaraj,
		\emph{On the finiteness of perfect powers in elliptic divisibility sequences},
		J. Th\'eor. Nombres Bordeaux \textbf{35} (2023), no. 1, 247--258.
		
		\bibitem{BiluHanrotVoutier2001}
		Y. Bilu, G. Hanrot, and P. M. Voutier,
		\emph{Existence of primitive divisors of Lucas and Lehmer numbers},
		J. Reine Angew. Math. \textbf{539} (2001), 75--122;
		with an appendix by M. Mignotte.
		
		\bibitem{Carmichael1913}
		R. D. Carmichael,
		\emph{On the numerical factors of the arithmetic forms \(\alpha^n \pm \beta^n\)},
		Ann. of Math. (2) \textbf{15} (1913), no. 1/4, 30--70.
		
		\bibitem{EverestMcLarenWard2006}
		G. Everest, G. McLaren, and T. Ward,
		\emph{Primitive divisors of elliptic divisibility sequences},
		J. Number Theory \textbf{118} (2006), no. 1, 71--89.
		
		\bibitem{EverestReynoldsStevens2007}
		G. Everest, J. Reynolds, and S. Stevens,
		\emph{On the denominators of rational points on elliptic curves},
		Bull. Lond. Math. Soc. \textbf{39} (2007), no. 5, 762--770.
		
		\bibitem{EverestPoortenShparlinskiWard2003}
		G. Everest, A. van der Poorten, I. Shparlinski, and T. Ward,
		\emph{Recurrence Sequences},
		Mathematical Surveys and Monographs, vol. 104,
		American Mathematical Society, Providence, RI, 2003.
		
		\bibitem{HajduLaishramSzikszai2016}
		L. Hajdu, S. Laishram, and M. Szikszai,
		\emph{Perfect powers in products of terms of elliptic divisibility sequences},
		Bull. Aust. Math. Soc. \textbf{94} (2016), no. 3, 395--404.
		
		\bibitem{NowrooziSiksek2024}
		M. Nowroozi and S. Siksek,
		\emph{Perfect powers in elliptic divisibility sequences},
		Bull. Lond. Math. Soc. \textbf{56} (2024), no. 11, 3331--3345.
		
		\bibitem{Reynolds2012}
		J. Reynolds,
		\emph{Perfect powers in elliptic divisibility sequences},
		J. Number Theory \textbf{132} (2012), no. 5, 998--1015.
		
		\bibitem{Silverman2005}
		J. H. Silverman,
		\emph{\(p\)-adic properties of division polynomials and elliptic divisibility sequences},
		Math. Ann. \textbf{332} (2005), no. 2, 443--471.
		
		\bibitem{Ward1948}
		M. Ward,
		\emph{Memoir on elliptic divisibility sequences},
		Amer. J. Math. \textbf{70} (1948), no. 1, 31--74.
		
	\end{thebibliography}
\end{document}